\documentclass[12pt]{article}
\usepackage{amssymb,amsmath}
\usepackage{amsthm}
\usepackage{cases}
\usepackage{amsfonts}
\usepackage{graphicx}
\usepackage{cite,color,xcolor}
\usepackage[left=2.7cm,right=2.7cm,top=2.4cm,bottom=2.4cm]{geometry}
\usepackage[colorlinks,citecolor=blue,urlcolor=blue]{hyperref}
\usepackage[utf8]{inputenc}

\newtheorem{theorem}{Theorem}[section]

\newtheorem{lemma}{Lemma}[section]
\newtheorem{proposition}{Proposition}[section]

\usepackage{txfonts}
\newcommand{\bal}{\begin{align}}
\newcommand{\bbal}{\begin{align*}}
\newcommand{\beq}{\begin{equation}}
\newcommand{\eeq}{\end{equation}}
\newcommand{\bca}{\begin{cases}}
\newcommand{\eca}{\end{cases}}

\newcommand{\pa}{\partial}
\newcommand{\fr}{\frac}

\newcommand{\dd}{\mathrm{d}}

\newcommand{\R}{\mathbb{R}}

\newcommand{\les}{\lesssim}

\newcommand{\f}{\left}
\newcommand{\g}{\right}

\begin{document}
\bibliographystyle{plain}
\title{Non-uniform convergence of solution for the Camassa--Holm equation in the zero-filter limit}

\author{Jinlu Li$^{1}$, Yanghai Yu$^{2,}$\footnote{E-mail: lijinlu@gnnu.edu.cn; yuyanghai214@sina.com(Corresponding author); mathzwp2010@163.com} and Weipeng Zhu$^{3}$\\
\small $^1$ School of Mathematics and Computer Sciences, Gannan Normal University, Ganzhou 341000, China\\
\small $^2$ School of Mathematics and Statistics, Anhui Normal University, Wuhu 241002, China\\
\small $^3$ School of Mathematics and Big Data, Foshan University, Foshan, Guangdong 528000, China}

\date{\today}
\maketitle\noindent{\hrulefill}

{\bf Abstract:} In this short note, we prove that given initial data $u_0\in H^s(\R)$ with $s>\fr32$ and for some $T>0$, the solution of the Camassa-Holm equation does not converges uniformly with respect to the initial data in $L^\infty(0,T;H^s(\R))$ to the inviscid Burgers equation as the filter parameter $\alpha$ tends to zero. This is a supplement to our recent result on the zero-filter limit.



\vskip0mm\noindent{\hrulefill}

\section{Introduction}

In this paper, we continue to consider the zero-filter limit $\alpha\to0$  for the Camassa--Holm equation in the Sobolev space
\begin{equation}\label{mC}
\begin{cases}
\pa_tm+2m\pa_xu+u\pa_xm=0, \quad (t,x)\in \R^+\times\R,\\
m=(1-\alpha^2\pa^2_x)u,\\
u(0,x)=u_0(x),
\end{cases}
\end{equation}
where  the constant $\alpha>0$ is a filter parameter.  When the filter parameter $\alpha=0$, Eq.\eqref{mC} becomes the Burgers equation
\begin{align}\label{bu}
\begin{cases}
u_t+3u \partial_xu=0,  \quad (t,x)\in \R^+\times\R,\\
u(0,x)=u_0(x).
\end{cases}
\end{align}
The Camassa--Holm equation was firstly proposed in the context of hereditary symmetries studied in \cite{Fokas} and then was derived explicitly as a water wave equation by Camassa--Holm \cite{Camassa}. \eqref{mC} is completely integrable \cite{Camassa,Constantin-P} with a bi-Hamiltonian structure \cite{Constantin-E,Fokas} and infinitely many conservation laws \cite{Camassa,Fokas}. Also, it admits exact peaked
soliton solutions (peakons) of the form $u(x,t)=ce^{-|x-ct|}$ with $c>0,$ which are orbitally stable \cite{Constantin.Strauss}. Another remarkable feature of the Camassa--Holm equation is the wave breaking phenomena: the solution remains bounded while its slope becomes unbounded in finite time \cite{Constantin,Escher2,Escher3}. It is worth mentioning that the peaked solitons present the characteristic for the travelling water waves of greatest height and largest amplitude and arise as solutions to the free-boundary problem for incompressible Euler equations over a flat bed, see Refs. \cite{Constantin-I,Escher4,Escher5} for the details.

We note that the pseudo-differential operator $(1-\alpha^2 \partial_x^2)^{-1}$  with the
Fourier multiplier $(1+\alpha^2 |\xi|^2)^{-1}$ can be defined as follows
\bal\label{g}
\quad\left(1-\alpha^2 \partial_x^2\right)^{-1} f=g*f,\quad \forall \;f \in L^2(\mathbb{R}),
\end{align}
where $g(x):=\frac{1}{2 \alpha} e^{-\frac{|x|}{\alpha}}$, $x \in \mathbb{R}$ and $*$ denotes convolution, then $u=g * m$. Using this identity and applying the pseudo-differential operator $\left(1-\alpha^2 \partial_x^2\right)^{-1}$ to Eq.\eqref{mC}, one can rewrite Eq.\eqref{mC} as a quasi-linear nonlocal evolution equation of hyperbolic type, namely
\begin{align}\label{c}
\begin{cases}
u_t+3u \partial_xu=-\alpha^2\partial^3_x \left(1-\alpha^2 \partial_x^2\right)^{-1}u^2-\frac{\alpha^2}{2}\partial_x \left(1-\alpha^2 \partial_x^2\right)^{-1} (\partial_xu)^2,\\
u(0,x)=u_0(x).
\end{cases}
\end{align}
Formally, as $\alpha\to0$, the solution of the Camassa--Holm equation \eqref{c} converges to the
solution of the following Burgers equation
\begin{align}\label{b}
\begin{cases}
u_t+3u \partial_xu=0, &\quad (t,x)\in \R^+\times\R,\\
u(0,x)=u_0(x).
\end{cases}
\end{align}
The Burgers equation is perhaps the most basic example of a PDE evolution leading to shocks. Further background and motivation for the  Burgers equation may be found in \cite{miao2009,Alibaud,Dong,Karch,Linares,Molinet} and references therein.

Gui-Liu \cite{GL} proved that the solutions of the  Camassa--Holm equation with additional dissipative term $\nu\Lambda^{\gamma}u$ does converge, at least locally, to the
one of the dissipative Burgers equation as the filter parameter $\alpha$ tends to zero in the lower regularity Sobolev spaces.
Recently,  in \cite{aml}, we considered the zero-filter limit for the Camassa–Holm equation without dissipative term in Sobolev spaces, and proved that the solution of \eqref{c} converges to the solution of the inviscid Burgers equation \eqref{b} in the topology of Sobolev spaces. Precisely speaking,
\begin{theorem}[\cite{aml}]\label{th1} Let $s>\frac32$ and $\alpha \in(0,1)$. Assume that the initial data $u_0\in H^s(\mathbb{R})$. Let $\mathbf{S}_{t}^{\mathbf{\alpha}}(u_0)$ and $\mathbf{S}_{t}^{0}(u_0)$ be the smooth solutions of \eqref{c} and \eqref{b} with the initial data $u_0$ respectively. Then there exists a time $T=T(\|u_0\|_{H^s})>0$ such that  $\mathbf{S}_{t}^{\mathbf{\alpha}}(u_0),\mathbf{S}_{t}^{0}(u_0)\in \mathcal{C}([0,T];H^s)$ and
$$
\lim_{\alpha\to 0}\left\|\mathbf{S}_{t}^{\mathbf{\alpha}}(u_0)-\mathbf{S}_{t}^{0}(u_0)\right\|_{L^\infty_TH^{s}}=0.
$$
\end{theorem}

An interesting problem appears:
For any $u_0\in U_R$, in the zero-filter limit $\alpha\to0$, whether or not the $H^{s}$-convergence
\begin{align*}
\mathbf{S}_{t}^\alpha(u_0)\rightarrow\mathbf{S}_{t}^0(u_0)\quad\text{in} \quad L_T^\infty H^s
\end{align*}
can be established {\it uniformly with respect to the initial data $u_0$?}

For any $R>0$, from now on, we denote any bounded subset $U_R\subset H^s(\mathbb{R})$ by
$$U_R:=\left\{\phi\in H^s(\mathbb{R}): \|\phi\|_{H^s(\mathbb{R})}\leq R\right\}.$$
In this paper, we shall answer the above question and state our main result as follows.
\begin{theorem}\label{th2} Let $\alpha\in (0,1)$ and $s>\fr32$. For any $u_0\in U_R$, let $\mathbf{S}_{t}^{\alpha}(u_0)$ and $\mathbf{S}_{t}^{0}(u_0)$ be the solutions of \eqref{c} and \eqref{b} with the same initial data $u_0$, respectively. Then a family of solutions $\f\{\mathbf{S}_{t}^{\alpha}(u_0)\g\}_{\alpha>0}$ to \eqref{c}
\begin{equation*}
\mathbf{S}_t^\alpha:\begin{cases}
U_R \rightarrow \mathcal{C}([0, T] ; H^s),\\
u_0\mapsto \mathbf{S}_t^\alpha(u_0),
\end{cases}
\end{equation*}
do not converge strongly in a uniform way with respect to initial data to the solution $\mathbf{S}_{t}^{0}(u_0)$ of \eqref{b} in $H^s$.
More precisely, there exists a sequence initial data $\{u^n_0\}_{n=1}^{\infty}\in U_R$ such that for a short time $T_0\leq T$
$$
\liminf_{\alpha_n\to 0}\left\|\mathbf{S}_{t}^{\alpha_n}(u^n_0)-\mathbf{S}_{t}^{0}(u^n_0)\right\|_{L^\infty_{T_0}H^s}\geq \eta_0,
$$
with some positive constant $\eta_0$. 
\end{theorem}

{\bf Notation}\; Throughout this paper, we will denote by $C$ any positive constant independent of the parameter $\alpha$, which may change from line to line. The symbol $\mathrm{A}\lesssim (\gtrsim)\mathrm{B}$ means that there is a uniform positive ``harmless" constant $C$ independent of $\mathrm{A}$ and $\mathrm{B}$ such that $\mathrm{A}\leq(\geq) C\mathrm{B}$, and we sometimes use the notation $\mathrm{A}\approx \mathrm{B}$ means that $\mathrm{A}\lesssim \mathrm{B}$ and $\mathrm{B}\lesssim \mathrm{A}$.
Given a Banach space $X$, we denote its norm by $\|\cdot\|_{X}$. For $I\subset\R$, we denote by $\mathcal{C}(I;X)$ the set of continuous functions on $I$ with values in $X$. Sometimes we will denote $L^p(0,T;X)$ by $L_T^pX$.
For $s\in\R$, the nonhomogeneous Sobolev space is defined by
$\|f\|^2_{H^s}=\int_{\R}(1+|\xi|^2)^s|\widehat{f}(\xi)|^2\dd \xi.$
We recall the classical result for later proof.
\begin{lemma}[\cite{B}]\label{le1}
For $s>0$, $H^s(\R)\cap L^\infty(\R)$ is an algebra.
Moreover, we have for any $u,v \in H^s(\R)\cap L^\infty(\R)$
\begin{align*}
&\|uv\|_{H^s(\R)}\leq C\big(\|u\|_{H^s(\R)}\|v\|_{L^\infty(\R)}+\|v\|_{H^s(\R)}\|u\|_{L^\infty(\R)}\big).
\end{align*}
In particular, for $s>\frac12$, due to the fact $H^s(\R)\hookrightarrow L^\infty(\R)$, then we have
\begin{align*}
&\|uv\|_{H^s(\R)}\leq C\|u\|_{H^s(\R)}\|v\|_{H^s(\R)}.
\end{align*}
\end{lemma}

\section{Proof of Theorem \ref{th2}}\label{sec3}

For fixed $\alpha>0$, by the classical local well-posedness result, we known that there exists a $T_\alpha=T(\|u_0\|_{H^s},s,\alpha)>0$ such that the Camassa-Holm has a unique solution $\mathbf{S}_{t}^{\mathbf{\alpha}}(u_0)\in\mathcal{C}([0,T_\alpha];H^s)$. Furthermore, we can obtain that $\exists\; T=T(\|u_0\|_{H^s},s)>0$ such that $T\leq T_{\alpha}$ and there exists $C_1>0$ independent of $\alpha$ such that
\begin{align}\label{m1}
\|\mathbf{S}_{t}^{\mathbf{\alpha}}(u_0)\|_{L_T^{\infty} H^s} \leq C_1\left\|u_0\right\|_{H^s}, \quad \forall \alpha \in[0,1).
\end{align}
Moreover, if $u_0 \in H^\gamma \cap H^s$ for some $\gamma\geq s-1$, then there exists $C_2(\left\|u_0\right\|_{H^s})>0$ independent of $\alpha$ such that
\begin{align}\label{m2}
\|\mathbf{S}_{t}^{\mathbf{\alpha}}(u_0)\|_{L_T^{\infty} H^\gamma} \leq C_2(\left\|u_0\right\|_{H^s})\left\|u_0\right\|_{H^\gamma} .
\end{align}
For more details on the proof of \eqref{m1} and \eqref{m2}, we can refer to see  \cite{aml}.

Next, we establish the following proposition will play a crucial role in the proof of Theorem \ref{th2}.
\begin{proposition}\label{pro1}
Let $\alpha\in[0,1)$. Assume that $s>\frac{3}{2}$ and $\|u_0\|_{H^s}\approx 1$. Let $\mathbf{S}_{t}^{\mathbf{\alpha}}(u_0)$ and $\mathbf{S}_{t}^{0}(u_0)$ be the smooth solutions of \eqref{c} and \eqref{b} with the same initial data $u_0$, respectively. Then we have
\bbal
\f\|\mathbf{S}^\alpha_{t}(u_0)-u_0-t\mathbf{E}(\alpha,u_0)\g\|_{H^{s}}\leq Ct^{2}\mathbf{F}(\alpha,u_0),
\end{align*}
where we denote
\bbal
&\mathbf{E}(\alpha,u_0):=-3u_0\pa_xu_0-\alpha^2\partial^3_x \left(1-\alpha^2 \partial_x^2\right)^{-1}u_0^2-\frac{\alpha^2}{2}\partial_x \left(1-\alpha^2 \partial_x^2\right)^{-1} (\partial_xu_0)^2,\\
&\mathbf{F}(\alpha,u_0):=\alpha\|u_0\|_{H^{s+1}}\f(\alpha\|u_0\|_{H^{s+1}}+\|u_0\|_{H^{s-1}}\|u_0\|_{H^{s+1}}\g)+ (\alpha+\|u_0\|_{H^{s-1}})\f(\|u_0\|_{H^{s+1}}+\|u_0\|_{H^{s-1}}\|u_0\|_{H^{s+2}}\g).
\end{align*}
\end{proposition}
\begin{proof} For simplicity, we denote $u(t)=\mathbf{S}^\alpha_t(u_0)$.
Firstly, we need to estimate the different Sobolev norms of the term $u(t)-u_0$, which can be bounded by $t$ multiplying the corresponding Besov norms of initial data $u_0$.
For $t\in[0,T]$, by the fundamental theorem of calculus in the time variable and using the product estimates from Lemma \ref{le1}, we obtain from \eqref{c} that
\bal\label{u1}
\|u(t)-u_0\|_{H^s}
&\leq \int^t_0\|\pa_\tau u\|_{H^s} \dd\tau
\nonumber\\&\leq \int^t_0\f(3\|u \pa_xu\|_{H^s}+\alpha^2\f\|\partial^3_x \left(1-\alpha^2 \partial_x^2\right)^{-1}u^2\g\|_{H^s}\g)\dd\tau\nonumber\\&\quad+\int^t_0\frac{\alpha^2}{2}\f\|\partial_x \left(1-\alpha^2 \partial_x^2\right)^{-1} (\partial_xu)^2\g\|_{H^s} \dd\tau
\nonumber\\&\lesssim t\f(\f\|u \pa_xu\g\|_{L_t^\infty H^s}+\alpha\f\|(\partial_xu)^2\g\|_{L_t^\infty H^s}\g)\nonumber\\
&\lesssim t\f(\|u\|_{L_t^\infty H^{s-1}}\|u\|_{L_t^\infty H^{s+1}}+\alpha\f\|\partial_xu\g\|_{L_t^\infty H^{s-1}}\f\|\partial_xu\g\|_{L_t^\infty H^s}\g)
\nonumber\\&\lesssim t\f(\|u_0\|_{H^{s-1}}\|u_0\|_{H^{s+1}}+\alpha\|u_0\|_{H^{s+1}}\g),
\end{align}
where we have used that $H^{s-1}(\R)\hookrightarrow L^\infty(\R)$ with $s>\frac{3}{2}$.

Following the same procedure of estimates as above, we have
\bal\label{u2}
\|u(t)-u_0\|_{H^{s-1}}
&\leq \int^t_0\|\pa_\tau u\|_{H^{s-1}} \dd\tau\nonumber\\
&\leq \int^t_0\f(3\|u \pa_xu\|_{H^{s-1}}+\alpha^2\f\|\partial^3_x \left(1-\alpha^2 \partial_x^2\right)^{-1}u^2\g\|_{H^{s-1}}\g)\dd\tau\nonumber\\
&\quad+\int^t_0\frac{\alpha^2}{2}\f\|\partial_x \left(1-\alpha^2 \partial_x^2\right)^{-1} (\partial_xu)^2\g\|_{H^{s-1}}  \dd\tau\nonumber\\
&\lesssim t\f(\f\|u \pa_xu\g\|_{L_t^\infty H^{s-1}}+\alpha\f\|(\partial_xu)^2\g\|_{L_t^\infty H^{s-1}}\g)\nonumber\\
&\lesssim t\f(\|u\|_{L_t^\infty H^{s-1}}\|u\|_{L_t^\infty H^s}+\alpha\|u\|^2_{L_t^\infty H^s}\g)
\nonumber\\&\lesssim t\f(\|u_0\|_{H^{s-1}}+\alpha\g),
\end{align}
\bal\label{u3}
\|u(t)-u_0\|_{H^{s+1}}
&\leq \int^t_0\|\pa_\tau u\|_{H^{s+1}} \dd\tau
\nonumber\\&\leq \int^t_0\f(3\|u \pa_xu\|_{H^{s+1}}+\alpha^2\f\|\partial^3_x \left(1-\alpha^2 \partial_x^2\right)^{-1}u^2\g\|_{H^{s+1}}\g)\dd\tau\nonumber\\
&\quad+\int^t_0\frac{\alpha^2}{2}\f\|\partial_x \left(1-\alpha^2 \partial_x^2\right)^{-1} (\partial_xu)^2\g\|_{H^{s+1}}  \dd\tau\nonumber\\
&\lesssim t\f(\f\|u \pa_xu\g\|_{L_t^\infty H^{s+1}}+\f\|(\partial_xu)^2\g\|_{L_t^\infty H^{s}}\g)\nonumber\\
&\lesssim t\f(\|u\|_{L_t^\infty H^{s-1}}\|u\|_{L_t^\infty H^{s+2}}+\|u\|_{L_t^\infty H^s}\|u\|_{L_t^\infty H^{s+1}}\g)
\nonumber\\&\lesssim t\f(\|u_0\|_{H^{s-1}}\|u_0\|_{H^{s+2}}+\|u_0\|_{H^{s+1}}\g).
\end{align}
Next, we estimate the $H^s$-norm for the term $u(t)-u_0-t\mathbf{E}(\alpha,u_0)$  which can be bounded by $t^2$
multiplying the Sobolev norms of initial data $u_0$. For $t\in[0,T]$, by the fundamental theorem of calculus in the time variable and using the product estimates from Lemma \ref{le1} again, we obtain from \eqref{c} that
\bbal
&\|u(t)-u_0-t\mathbf{E}(\alpha,u_0)\|_{H^s}
\leq \int^t_0\|\pa_\tau u-\mathbf{E}(\alpha,u_0)\|_{H^s} \dd\tau
\\&\leq \int^t_0\f(3\f\|u\pa_xu-u_0\pa_xu_0\g\|_{H^s} +\alpha^2\f\|\partial^3_x \left(1-\alpha^2 \partial_x^2\right)^{-1}\f(u^2-u_0^2\g)\g\|_{H^{s}}\g)\dd\tau\\
&\quad+ \int^t_0\frac{\alpha^2}{2}\f\|\partial_x \left(1-\alpha^2 \partial_x^2\right)^{-1} \f((\partial_xu)^2-(\partial_xu_0)^2\g)\g\|_{H^s} \dd\tau\\
&\les \int^t_0\f(\f\|u^2-u_0^2\g\|_{H^{s+1}} +\alpha\f\|(\partial_xu)^2-(\partial_xu_0)^2\g\|_{H^s}\g) \dd\tau
\\&\les \int^t_0\f(\|u(\tau)-u_0\|_{H^{s-1}} \|u_0\|_{H^{s+1}}+\|u_0\|_{H^{s-1}}\|u(\tau)-u_0\|_{H^{s+1}}\g)\dd \tau\\
&\quad+ \int^t_0\alpha\f(\|u_0\|_{H^{s+1}}\|u(\tau)-u_0\|_{H^s}+\|u(\tau)-u_0\|_{H^{s+1}}\g)\dd \tau
\\&\les t^2\mathbf{F}(\alpha,u_0),
\end{align*}
where we have used \eqref{u1}-\eqref{u3} in the last step.
Thus, we complete the proof of Proposition \ref{pro1}.
\end{proof}
Before proving Theorem \ref{th2}, we need to construct the initial data $u_0$. Firstly, we introduce smooth, radial cut-off functions to localize the frequency region. Precisely,
let $\widehat{\phi}\in \mathcal{C}^\infty_0(\mathbb{R})$ be an even, real-valued and non-negative function on $\R$ and satisfy
\begin{numcases}{\widehat{\phi}(\xi)=}
1,&if $|\xi|\leq \frac{1}{4}$,\nonumber\\
0,&if $|\xi|\geq \frac{1}{2}$.\nonumber
\end{numcases}
Motivated by \cite{Lyz}, we establish the following crucial lemmas which will be used later on.
\begin{lemma}\label{le2} Let $s\in\R$.
Define the high frequency function $f_n$ and the low frequency function $g_n$ by
\bbal
&f_n=2^{-ns}\phi(x)\sin \f(\frac{17}{12}2^nx\g),\quad n\gg1,\\
&g_n=2^{-n}\phi(x).
\end{align*}
Then for any $\sigma\in\R$, there exists a positive constant $C=C(\phi)$ such that
\bbal
&\|f_n\|_{L^\infty}\leq C2^{-ns},\quad\|\pa_xf_n\|_{L^\infty}\leq C2^{-n(s-1)},\\
&\|f_n\|_{H^\sigma}\approx2^{n(\sigma-s)},\quad \|g_n\|_{H^\sigma}\approx2^{-n},\\
&\liminf_{n\rightarrow \infty}\|g_n\pa_xf_n\|_{H^s}\geq C.
\end{align*}
\end{lemma}
\begin{lemma}\label{le3} Let $f_n$ and $g_n$ be defined as Lemma \ref{le2}. Assume that $s\in\R$ and $\alpha_n=2^{-n}$. Then there exists a positive constant $C=C(\phi)$ such that
\bbal
&\liminf_{n\rightarrow \infty}\f\|\alpha^2_n\partial^2_x \left(1-\alpha^2_n \partial_x^2\right)^{-1}(g_n\pa_xf_n)\g\|_{H^s}\geq C.
\end{align*}
\end{lemma}
\begin{proof} By the construction of $f_n$ and $g_n$, one has (for more details see \cite{Lyz})
\bbal
\mathrm{supp}\ \widehat{g_n\pa_xf_n}\subset \f\{\xi\in\R: \ \frac{17}{12}2^n-1\leq |\xi|\leq \frac{17}{12}2^n+1\g\}.
\end{align*}
By the Plancherel's identity, we deduce that
\bbal
\f\|\alpha^2_n\partial^2_x \left(1-\alpha^2_n \partial_x^2\right)^{-1}(g_n\pa_xf_n)\g\|_{H^s}\approx\|g_n\pa_xf_n\|_{H^s}
\end{align*}
Using Lemma \ref{le2}, we complete the proof of Lemma  \ref{le3}.
\end{proof}
Now we begin to prove our main Theorem \ref{th2}.
Letting $\alpha_n=2^{-n}$.
We set the initial data  $u^n_0=f_n+g_n$.
It is easy to show that
\bbal
\|f_n\|_{H^{s+kn}}\lesssim 2^{kn}\quad\text{and}\quad \|g_n\|_{H^{s+kn}}\lesssim 2^{-n}\quad\text{for}\quad k\in\{-1,0,1,2\},
\end{align*}
which gives directly that
\bbal
\|u^n_0\|_{H^{s+kn}}\lesssim 2^{kn}.
\end{align*}
Thus
\bbal
\mathbf{F}(0,u^n_0)\lesssim  1\quad\text{and}\quad\mathbf{F}(\alpha_n,u^n_0)\lesssim  1.
\end{align*}
We decompose the solution $\mathbf{S}^{\alpha}_t(u_0)$ to \eqref{c} and the solution $\mathbf{S}^{0}_t(u_0)$ to \eqref{b} into three parts, respectively
\bbal
&\mathbf{S}^{\alpha_n}_{t}(u^n_0)=u^n_0+\underbrace{\mathbf{S}^{\alpha_n}_{t}(u^n_0)-u_0-t\mathbf{E}(\alpha_n,u^n_0)}_{=:\,\mathbf{I}_1}+t\mathbf{E}(\alpha_n,u^n_0),\\
&\mathbf{S}^{0}_{t}(u^n_0)=u^n_0+\underbrace{\mathbf{S}^{0}_{t}(u^n_0)-u_0-t\mathbf{E}(0,u^n_0)}_{=:\,\mathbf{I}_2}+t\mathbf{E}(0,u^n_0)\quad\text{and}\\
&\mathbf{E}(\alpha_n,u^n_0)-\mathbf{E}(0,u^n_0)=-2\alpha^2_n\partial^2_x \left(1-\alpha^2_n \partial_x^2\right)^{-1}(u^n_0\pa_xu^n_0)-\frac{\alpha^2_n}{2}\partial_x \left(1-\alpha^2_n \partial_x^2\right)^{-1} (\partial_xu^n_0)^2.
\end{align*}
Furthermore, notice that
$
u^n_{0}\pa_xu^n_{0}=g_n\pa_xf_n+f_n\pa_xf_n+u^n_{0}\pa_xg_n,
$
by Proposition \ref{pro1}, we deduce that
\bal\label{yyh}
&\f\|\mathbf{S}^{\alpha_n}_{t}(u^n_0)-\mathbf{S}^0_{t}(u^n_0)\g\|_{H^s}
\geq~t\f\|\mathbf{E}(\alpha_n,u^n_0)-\mathbf{E}(0,u^n_0)\g\|_{H^s}-\f\|\mathbf{I}_1\g\|_{H^s}-\f\|\mathbf{I}_2\g\|_{H^s}\nonumber\\
\geq&~ t\f(2\f\|\alpha^2_n\partial^2_x \left(1-\alpha^2_n \partial_x^2\right)^{-1}(u^n_0\pa_xu^n_0)\g\|_{H^s}
-\frac{\alpha^2_n}{2}\f\|\partial_x \left(1-\alpha^2_n \partial_x^2\right)^{-1} (\partial_xu_0)^2\g\|_{H^s}\g)-Ct^{2}\nonumber\\
\gtrsim&~ t\f(\f\|\alpha^2_n\partial^2_x \left(1-\alpha^2_n \partial_x^2\right)^{-1}(g_n\pa_xf_n)\g\|_{H^s}
-\f\|f_n\pa_xf_n\g\|_{H^s}-\f\|u^n_{0}\pa_xg_n\g\|_{H^s}-2^{-n}\f\|(\partial_xu^n_0)^2\g\|_{H^s}\g)-t^{2}.
\end{align}
Using Lemma \ref{le1} and Lemma \ref{le2}, after simple calculation, we obtain
\bbal
&\f\|f_n\pa_xf_n\g\|_{H^s}\les\|f_n\|_{L^\infty}\|f_n\|_{H^{s+1}}\les2^{-n(s-1)},\\
&\f\|u^n_{0}\pa_xg_n\g\|_{H^s}\les\|u^n_0\|_{H^s}\|g_n\|_{H^{s+1}}\les2^{-n},\\
&\f\|(\partial_xu^n_0)^2\g\|_{H^s}\les\|\partial_xu^n_0\|_{L^\infty}\|u^n_0\|_{H^{s+1}}\les2^n(2^{-n}+2^{-n(s-1)})\les1+2^{-n(s-2)}.
\end{align*}
Plugging the above estimates into \eqref{yyh} and combining Lemma \ref{le3} yields that
\bbal
\liminf_{n\rightarrow \infty}\f\|\mathbf{S}^{\alpha_n}_{t}(u^n_0)-\mathbf{S}^0_{t}(u^n_0)\g\|_{H^s}\gtrsim t\quad\text{for} \ t \ \text{small enough}.
\end{align*}
This completes the proof of Theorem \ref{th2}.

\section*{Acknowledgements}
J. Li is supported by the National Natural Science Foundation of China (12161004), Training Program for Academic and Technical Leaders of Major Disciplines in Jiangxi Province (20232BCJ23009) and Jiangxi Provincial Natural Science Foundation (20224BAB201008). Y. Yu is supported by the National Natural Science Foundation of China (12101011). W. Zhu is supported by the National Natural Science Foundation of China (12201118) and Guangdong Basic and Applied Basic Research Foundation (2021A1515111018).

\section*{Declarations}
\noindent\textbf{Data Availability} No data was used for the research described in the article.

\vspace*{1em}
\noindent\textbf{Conflict of interest}
The authors declare that they have no conflict of interest.

\end{document}